\documentclass[12pt]{amsart}

\usepackage{amsmath, amsthm, amssymb}
\usepackage{url} 
\usepackage{graphicx}
\usepackage{xcolor}
\usepackage{moreverb}

\usepackage{fancyvrb}

\def\r{\mathbb R}

\newtheorem{theorem}{Theorem}[section]
 \newtheorem{proposition}[theorem]{Proposition}
 \newtheorem{corollary}[theorem]{Corollary}
\theoremstyle{definition}
\newtheorem{definition}[theorem]{Definition}
\newtheorem{example}[theorem]{Example}
\newtheorem{remark}[theorem]{Remark}

\begin{document}

\title[Surfaces with constant sectional curvature]{Constant sectional curvature surfaces  with a semi-symmetric non-metric connection}

\author{Muhittin Evren Aydin}
\address{Department of Mathematics, Faculty of Science, Firat University, Elazig,  23200 Turkey}
\email{meaydin@firat.edu.tr}
\author{Rafael L\'opez}
\address{Departamento de Geometr\'{\i}a y Topolog\'{\i}a Universidad de Granada 18071 Granada, Spain}
\email{rcamino@ugr.es}
\author{ Adela Mihai}
 \address{Technical University of Civil Engineering Bucharest,
Department of Mathematics and Computer Science, 020396, Bucharest, Romania
and Transilvania University of Bra\c{s}ov, Interdisciplinary Doctoral
School, 500036, Bra\c{s}ov, Romania}
 \email{adela.mihai@utcb.ro, adela.mihai@unitbv.ro}

\begin{abstract}
Consider the Euclidean space $\r^3$ endowed with a canonical semi-symmetric non-metric connection determined by a vector field $\mathsf{C}\in\mathfrak{X}(\r^3)$. We study surfaces when the sectional curvature with respect to this connection is constant. In case that the surface is cylindrical, we obtain full classification when the rulings are orthogonal or parallel to $\mathsf{C}$. If the surface is rotational, we prove that  the rotation axis is parallel to $\mathsf{C}$ and we classify all conical rotational surfaces with constant sectional curvature. Finally, for the particular case  $\frac12$ of the sectional curvature, the existence of rotational surfaces orthogonally intersecting the rotation axis is also obtained.
\end{abstract}
\maketitle

\section{Introduction} \label{intro}

Friedmann and Schouten introduced in 1924 the notion of a semi-symmetric connection in a Riemannian manifold \cite{fs}. An affine connection $\widetilde{\nabla}$  in a Riemannian manifold $(\widetilde{M},\tilde{g})$ is said to be {\it semi-symmetric connection} if there is a non-zero vector field $\mathsf{C}\in\mathfrak{X}(\widetilde{M}) $ such that its torsion $T$ satisfies the identity   
\begin{equation} \label{tors}
\widetilde{T}(X,Y)=\tilde{g}(\mathsf{C},Y) X-\tilde{g}(\mathsf{C},X) Y, \quad   X,Y \in \mathfrak{X}(\widetilde{M}).
\end{equation}
  If in addition $\widetilde{\nabla}\tilde{g}= 0$, the connection  $\widetilde{\nabla}$ is called a {\it semi-symmetric metric connection} \cite{hay}. Yano studied  semi-symmetric metric connections with zero curvature and when  the covariant derivative of the torsion tensor vanishes \cite{ya}. Submanifolds of Riemannian manifolds with semi-symmetric metric connections have been also  investigated: without aiming a complete list, we refer to the readers to \cite{im,lyl,mc1,mc2,na,wa}.

If $\widetilde{\nabla}\tilde{g}\not=0$, the connection is called {\it semi-symmetric non-metric connection} (snm-connection to abbreviate) \cite{ag,ac}. In this case, there is a relation between $\widetilde{\nabla}$ and the Levi-Civita connection $\widetilde{\nabla}^0$   of $(\widetilde{M},\tilde{g})$, namely,  
\begin{equation} \label{nabla0}
\widetilde{\nabla}_XY=\widetilde{\nabla}^0_XY+\tilde{g}(\mathsf{C},Y) X, \quad   X,Y \in \mathfrak{X}(\widetilde{M}).
\end{equation}
Such as it occurs for semi-symmetric metric connections, it is natural to study submanifolds of Riemannian manifolds endowed with a snm-connection $\widetilde{\nabla}$. Let   $M$ be a submanifold of $\widetilde{M}$. Denote by  $\nabla$ (resp. $\nabla^0$)   the induced connection on $M$ by $\widetilde{\nabla}$ (resp. $\widetilde{\nabla}^0$). The Gauss formulas  are given by
\begin{equation*} 
\begin{split}
\widetilde{\nabla}_XY&=\nabla_XY+h(X,Y),\\
\widetilde{\nabla}^0_XY&=\nabla^0_XY+h^0(X,Y),
\end{split}
\end{equation*}
 for all $ X,Y \in \mathfrak{X}(M)$, where   $h$ is a $(0,2)$-tensor field on $M$ and $h^0$ is the second fundamental form of $M$. It is known  that $h=h^0$    \cite{ac}. Hence that problems of extrinsic nature are the same one that for the Levi-Civita connection.

We consider intrinsic geometry of submanifolds. One of the main concepts in intrinsic Riemannian geometry is that of  sectional curvature. It is natural to carry this concept for snm-connections. However, the sectional curvature of $(\widetilde{M},\tilde{g})$ with respect to $\widetilde{\nabla}$ cannot be defined by the usual way as the Levi-Civita connection $\widetilde{\nabla}^0$. This is because if $\widetilde{R}$ is the curvature tensor of $\widetilde{\nabla}$, the quantity $\tilde{g}(\widetilde{R}(e_1,e_2)e_2,e_1)$, where $\{e_1,e_2\}$ is an orthonormal basis of $\pi$, depends on the basis $\{e_1,e_2\}$: see Sect. \ref{s2} for details. In contrast, the third author of this paper, jointly with I. Mihai,  proved that $\tilde{g}(\widetilde{R}(e_1,e_2)e_2,e_1)+\tilde{g}(\widetilde{R}(e_2,e_1)e_1,e_2)$ is independent on the basis  \cite{am0}. Then they introduced the following notion of sectional curvature for snm-connections.

\begin{definition}   \label{def-sc}
Let $(\widetilde{M},\tilde{g})$ be a Riemannian manifold endowed with a snm-connection $\widetilde{\nabla}$. If $\pi$ is a plane in $ T_p\widetilde{M}$ with an orthonormal basis $\{e_1,e_2\}$, then the {\it sectional curvature} of $\pi$ with respect to $\widetilde{\nabla}$ is defined by
\begin{equation}\label{sc}
\widetilde{K}(\pi)=\frac{\tilde{g}(\widetilde{R}(e_1,e_2)e_2,e_1)+\tilde{g}(\widetilde{R}(e_2,e_1)e_1,e_2)}{2}.
\end{equation}
\end{definition}

Once we have the notion of sectional curvature, it is natural to ask for those submanifolds with  constant sectional curvature.  As for the Levi-Civita connection,  this question is difficult to address in all its generality.

 In this paper, we   consider that   the ambient space is the $3$-dimensional Euclidean space $\r^3$ endowed with the Euclidean metric $\langle,\rangle$. The amount of snm-connections of $\r^3$ is given by the vector fields $\mathsf{C}$ in the definition \eqref{tors} of a semi-symmetric connection.   One of the simplest choices of snm-connections of $\r^3$ is that $\mathsf{C}$ is a canonical vector field.  To be precise, let $(x,y,z)$ be canonical coordinates of $\r^3$ and let  $\{\partial_x,\partial_y,\partial_z\}$ be the corresponding  basis of $\mathfrak{X}(\r^3)$. In fact, if the vector field $\mathsf{C}$ is assumed to be canonical, namely $\mathsf{C} \in \{\partial_x,\partial_y,\partial_z\}$ then, after a change of coordinates of $\r^3$,  $\mathsf{C}$ is a unit constant vector field. 
 
\begin{definition}
A snm-connection $\widetilde{\nabla}$ on $\r^3$ is said to be {\it canonical} if $\mathsf{C}\in\mathfrak{X}(\r^3)$ is a unit constant vector field.
\end{definition}

From now on, unless otherwise specified, we denote by $\mathsf{C}$ a unit constant vector field on $\r^3$.

Definitively, the problem that we study is the classification of surfaces with constant sectional curvature for a given canonical snm-connection $\widetilde{\nabla}$ of $\r^3$. A way to tackle this problem is to impose a certain geometric condition on the surface. A natural   condition is that the surface is invariant by a one-parameter group of rigid motions. Denote $K$ by the sectional curvature with respect to the induced connection on the surface from $\widetilde{\nabla}$. Assuming a certain invariance of the surface, it allows us to expect that the equation $K=c$ can be expressed as an ordinary differential equation, where, under mild conditions, the existence is assured. For example, we can assume that the surface is invariant by a group of translations  or that the surface is invariant by a group of rotations. In the first case, the surface is called cylindrical and in the second one, rotational surface, or surface of revolution.

The organization of this paper is according to both types of surfaces. In Sect. \ref{s2} we prove an useful formula for computing the sectional curvature $K$ of a surface in terms of that of $\r^3$ and the Gaussian and mean curvatures of the surface. We will show some explicit examples of computations of sectional curvatures. 

Section \ref{s3} is devoted to cylindrical surfaces.  A cylindrical surface  can be parametrized by $\psi(s,t)=\gamma(s)+t\vec{w}$, $s\in I\subset\r$, $t\in\r$,  where   $\vec{w}\in\r^3$ is a unitary vector and $\gamma\colon I\to\r^3$ is a curve contained in a plane orthogonal to $\vec{w}$. The surface is invariant by the group of translations generated by $\vec{w}$.  After computing the sectional curvature $K$ in Thm. \ref{cyl-t2}, in Cor. \ref{cor32}, we prove that any cylindrical surface whose rulings are parallel to $\mathsf{C}$ has constant sectional curvature $K$, being $K=\frac12$. Another interesting case of cylindrical surfaces is that the rulings   are   orthogonal  to $\mathsf{C}$.  We obtain  a full classification of these cylindrical surfaces with $K$ constant depending on the sign of $K$  (Cor. \ref{c-cyl}). For the particular values $K=1/2$ and $K=-1/2$, in Cor. \ref{cor-34} we obtain explicit parametrizations of the surfaces. 
 
Rotational surfaces are invariant by rotations about an axis $L$ of $\r^3$ and such surfaces with $K$ constant will be studied in Sect. \ref{s4}. It is worth to point out that there is not a {\it priori} relation between the axis $L$ and the vector field $\mathsf{C}$ that defines the canonical snm-connection. However, we prove in Thm. \ref{t1} that $L$ and $\mathsf{C}$ must be parallel. In Thm. \ref{t-s}, we classify all conical rotational surfaces with $K$ constant proving that these surfaces are planes or circular cylinders. As a last observation, when $K=\frac12$, in Thm. \ref{ort}, the existence of rotational surfaces orthogonally intersecting the rotation axis is also obtained.

\section{Preliminaries} \label{s2}

Let $(\widetilde{M},\tilde{g})$ be a Riemannian manifold of dimension $\geq 2$ and let $\widetilde{\nabla}$ be an affine connection on $\widetilde{M}$. The torsion and curvature of $\widetilde{\nabla}$ are respectively a $(1,2)$-tensor field $\widetilde{T}$ and a $(1,3)$-tensor field $\widetilde{R}$ defined by
\begin{equation*}
\begin{split}
\widetilde{T}(X,Y)&=\widetilde{\nabla}_XY-\widetilde{\nabla}_YX-[X,Y], \\
\widetilde{R}(X,Y)Z&=\widetilde{\nabla}_X\widetilde{\nabla}_YZ-\widetilde{\nabla}_Y\widetilde{\nabla}_XZ-\widetilde{\nabla}_{[X,Y]}Z, 
\end{split}
\end{equation*}
 for $    X,Y,Z \in \mathfrak{X}(\widetilde{M})$. Let  $\widetilde{\nabla}$ be   a snm-connection on $(\widetilde{M},\tilde{g})$ determined by a vector field $\mathsf{C}\in\mathfrak{X}(\widetilde{M})$. Using   \eqref{nabla0},  there is also a relation between $\widetilde{R}$ and the Riemannian curvature tensor $\widetilde{R}^0$ of $\widetilde{\nabla}^0$ (\cite{ag,am0}). Indeed, for orthonormal vectors $e_1,e_2\in T_p\widetilde{M}$, $p\in\widetilde{M}$, we have 
$$
\tilde{g}(\widetilde{R}(e_1,e_2)e_2,e_1)=\tilde{g}(\widetilde{R}^0(e_1,e_2)e_2,e_1)-e_2(\tilde{g}(\mathsf{C},e_2))+\tilde{g}(\mathsf{C},\widetilde{\nabla}^0_{e_2}e_2)+\tilde{g}(\mathsf{C},e_2)^2. 
$$
 Although the first term at the right hand-side is the sectional curvature of the plane section $\pi=\text{span}\{e_1,e_2\}$, the term at the left hand-side depends on the choice of the basis of $\pi$. Therefore, the value $ \tilde{g}(\widetilde{R}(e_1,e_2)e_2,e_1)$ does not stand for a sectional curvature. The     quantity  \eqref{sc} was proposed in \cite{am0} as the definition of sectional curvature of $\pi$ with respect to $\widetilde{\nabla}$ because it is independent on the basis in $T_p\widetilde{M}$.   In case that $\{e_1,e_2\}$ is an arbitrary basis of $\pi$, it is immediate to see 
 \begin{equation}\label{k2}
\widetilde{K}(\pi)=\frac{\tilde{g}(\widetilde{R}(e_1,e_2)e_2,e_1)+\tilde{g}(\widetilde{R}(e_2,e_1)e_1,e_2)}{2( \tilde{g}(e_1,e_1)\tilde{g}(e_2,e_2)-\tilde{g}(e_1,e_2)^2 )}.
\end{equation}

From now on, suppose that  $\widetilde{M}$ is the Euclidean space $\r^3$. We compute the sectional curvature of a plane of $\r^3$. 


\begin{proposition}\label{pr-21}
Let $\widetilde{\nabla}$ be a canonical snm-connection  on $\r^3$. If $\pi$ is a plane of $\r^3$, then its sectional curvature is 
$$
\widetilde{K}(\pi)=\frac{\langle \vec{u},\mathsf{C}\rangle^2+\langle \vec{v},\mathsf{C}\rangle^2}{2},
$$
where $\{\vec{u},\vec{v}\}$ is an orthonormal basis of $\pi$.  As a consequence, $\widetilde{K}(\pi)$ is constant with $0\leq \widetilde{K}(\pi)\leq \frac12  $. Furthermore,  $\widetilde{K}(\pi)=0$ (resp. $\widetilde{K}(\pi)= \frac12$) if and only if  $\pi$ is perpendicular to $\mathsf{C}$ (resp. $\pi$ is parallel to $\mathsf{C}$). 
\end{proposition}

\begin{proof}
Using \eqref{nabla0} we compute
\begin{eqnarray*}
&\widetilde{\nabla}_{\vec{u}}\vec{u} = \langle \vec{u},\mathsf{C}\rangle \vec{u}, & \widetilde{\nabla}_{\vec{u}}\vec{v}=\langle \vec{v},\mathsf{C}\rangle \vec{u}, \\
&\widetilde{\nabla}_{\vec{v}}\vec{u}=\langle \vec{u},\mathsf{C}\rangle \vec{v}, &
\widetilde{\nabla}_{\vec{v}}\vec{v}=\langle \vec{v},\mathsf{C}\rangle \vec{v},
\end{eqnarray*}
and
\begin{eqnarray*}
&\widetilde{\nabla}_{\vec{u}}\widetilde{\nabla}_{\vec{v}}\vec{v}=\langle \vec{v},\mathsf{C}\rangle^2 \vec{u}, & \widetilde{\nabla}_{\vec{v}}\widetilde{\nabla}_{\vec{u}}\vec{v}=\langle \vec{u},\mathsf{C}\rangle\langle \vec{v},\mathsf{C}\rangle \vec{v}, \\
&\widetilde{\nabla}_{\vec{v}}\widetilde{\nabla}_{\vec{u}}\vec{u}=\langle \vec{u},\mathsf{C}\rangle^2 \vec{v}, &
\widetilde{\nabla}_{\vec{u}}\widetilde{\nabla}_{\vec{v}}\vec{u}=\langle \vec{u},\mathsf{C}\rangle\langle \vec{v},\mathsf{C}\rangle \vec{u}.
\end{eqnarray*}
Also it is easy to see $[\vec{u},\vec{v}]=0$. Hence the curvature tensor $\widetilde{R}$ is determined by 
\begin{equation*}
\begin{split}
\widetilde{R}(\vec{u},\vec{v})\vec{v}&=\langle \vec{v},\mathsf{C}\rangle^2 \vec{u}-\langle \vec{u},\mathsf{C}\rangle\langle \vec{v},\mathsf{C}\rangle \vec{v}, \\
\widetilde{R}(\vec{v},\vec{u})\vec{u}&=\langle \vec{u},\mathsf{C}\rangle^2 \vec{v}-\langle \vec{u},\mathsf{C}\rangle\langle \vec{v},\mathsf{C}\rangle \vec{u}.
\end{split}
\end{equation*}
This gives the formula for $\widetilde{K}(\pi)$.   The last statement is a consequence of this formula.  
\end{proof}

 \begin{remark}
The notion of scalar curvature at a point $p\in\r^3$ with respect to a snm-connection $\widetilde{\nabla}$ can be introduced in a similar manner as for the Levi-Civita connection. Let $\{\vec{u},\vec{v},\vec{w}\}$ be an orthonormal basis of $T_p\r^3$, $p\in\r^3$. The {\it scalar curvature} $\rho$  with respect to $\widetilde{\nabla}$ is defined by
$$
\rho(p)=\widetilde{K}(\vec{u},\vec{v})+\widetilde{K}(\vec{u},\vec{w})+\widetilde{K}(\vec{v},\vec{w}) , \quad p\in \r^3.
$$
If $\widetilde{\nabla}$ is canonical, then by Prop. \ref{pr-21} the scalar curvature is constant, namely $\rho(p)=1$, for every $p\in \r^3$.  
\end{remark}

We conclude this section establishing a relation between the sectional curvatures $K$ and $\widetilde{K}$ of a surface in $\r^3$ in terms of the Gaussian and the mean curvatures of the surface.   Let $M$ be an oriented surface immersed   in $\r^3$ and $N$ its unit normal vector field. Let also $\widetilde{\nabla}$ be a snm-connection on $\r^3$ determined by an arbitrary vector field $\mathsf{C}\in\mathfrak{X}(\r^3)$. We have the decomposition of  $\mathsf{C}$ in its tangential and normal components with respect to   $M$,
$$
\mathsf{C}=\mathsf{C}^{\top}+\langle \mathsf{C},N\rangle N.
$$
If $X,Y,Z,U\in\mathfrak{X}(M)$, then the  Gauss equation with respect to $\widetilde{\nabla}$ is (\cite{ac}):
\begin{equation} \label{gauss-eq}
\begin{split}
 \langle R(X,Y)Z,U \rangle&= \langle \widetilde{R}(X,Y)Z,U\rangle-\langle h(X,Z),h(Y,U)\rangle +\langle h(X,U),h(Y,Z)\rangle  \\
&  +\langle \mathsf{C},N\rangle\left( \langle h(X,Z),N\rangle\langle Y,U\rangle- \langle h(Y,Z),N\rangle \langle X,U\rangle\right).
\end{split}%
\end{equation}

\begin{proposition} Let $M$ be an oriented surface in $\r^3$ and denote by  $G$ and $H$ the Gaussian curvature and the mean curvature of $M$, respectively, with respect to the Levi-Civita connection. Then  
\begin{equation}\label{help0}K=\widetilde{K}+G-\langle \mathsf{C},N\rangle H.
\end{equation}
Moreover, if  $p\in M$ then  there is an orthogonal basis $\{e_1,e_2\}$ of $T_pM$   such that
\begin{equation}\label{help}
\begin{split}
\langle R(e_1,e_2)e_2,e_1 \rangle &= \langle \widetilde{R}(e_1,e_2)e_2,e_1 \rangle+h_{11}h_{22}-g_{11}h_{22}\langle \mathsf{C},N\rangle,\\
\langle R(e_2,e_1)e_1,e_2 \rangle &=\langle \widetilde{R}(e_2,e_1)e_1,e_2 \rangle+h_{11}h_{22}-g_{22}h_{11}\langle \mathsf{C},N\rangle,
\end{split}
\end{equation}
where $g_{ij}=\langle e_i,e_j\rangle$ and $h_{ij}$ are the coefficients of the second fundamental form $h$.
\end{proposition}

\begin{proof}
Since the codimension of $M$ in $\r^3$ is $1$, it has trivially flat normal bundle. Let $\{e_1,e_2\}$ be an orthogonal basis  of $T_pM$ such that $g_{12}=0$ and $h_{12}^0=0$, where $h_{ij}^0$ are the coefficients of the second fundamental form of $M$ with respect to the Levi-Civita connection: see \cite[Props. 3.1 and 3.2]{che}.  Therefore we have $h_{12}=0$ because $h=h^0$. By the Gauss equation \eqref{gauss-eq} we obtain \eqref{help}. With respect to this basis, we have 
$$G=\frac{h_{11}h_{22}}{g_{11}g_{22}},\quad H=\frac{g_{22}h_{11}+g_{11}h_{22}}{2g_{11}g_{22}}.$$
Identity \eqref{help0} is a consequence of \eqref{help} and formulas \eqref{k2} for   $K$ and $\widetilde{K}$. 
\end{proof}

\begin{remark}\label{re-23}
Identity \eqref{help0} is satisfied for any vector field $\mathsf{C}\in\mathfrak{X}(\r^3)$. Notice also that $K$ is invariant by translations of $\r^3$. This is because $G$ and $H$ do no change, as well as $\widetilde{K}$ because a plane $\pi$ is not affected by translations. However, rigid motions change the value of $\widetilde{K}$ and, consequently of $K$. This is because of the presence of the vector field $\mathsf{C}$ in \eqref{nabla0}  for computing the successive covariant derivatives.
\end{remark}

Simple consequences of the relation \eqref{help0} appear in the following result.

\begin{corollary}
\begin{enumerate}
\item For a plane, we have $K=\widetilde{K}$. In particular, $K\geq 0$ and equality holds if and only if the plane is orthogonal to the vector field $\mathsf{C}$. 
\item For a cylindrical surface whose rulings are parallel  to $\mathsf{C}$, we have $K=\widetilde{K}$.
\end{enumerate}
\end{corollary}

\begin{proof} It is immediate because for a plane we have $G=H=0$, and for a cylindrical surface with rulings parallel to $\mathsf{C}$ we have $G=0$ and $\langle \mathsf{C},N\rangle=0$.
\end{proof}

Thanks to this corollary we see that a plane and a cylindrical cylinder satisfy the equality $K=\widetilde{K}$. 
In general, a surface satisfies $K=\widetilde{K}$ if and only if  $G=\langle \mathsf{C},N\rangle H$. In case that $\mathsf{C}$ is a canonical vector field, we construct such a surface as follows.  

\begin{example}\label{ex}
Let $\mathsf{C}=\partial_z$. To find a surface satisfying $G=\langle \partial_z ,N \rangle H$, we consider surfaces that are graphs of smooth functions   $z=u(x,y)$, where   $(x,y)\in \Omega \subset \r^2$. Then it is not difficult to find that the relation $G=\langle N,\partial_z\rangle H$ is written by 
$$
2(u_{xx}u_{yy}-u_{xy}^2)=(1+u_{y}^2)u_{xx}-2u_{x}u_{y}u_{xy}+(1+u_{x}^2)u_{yy}.
$$
We find solutions of this equation by   the technique of separation of variables. Assuming $u(x,y)=f(x)+g(y)$, for smooth functions $f=f(x)$ and $g=g(y)$, $x\in I\subset\r$, $y\in J\subset\r$, the above equation becomes
\begin{equation} \label{ex-1}
2f''g''=f''(1+g'^2)+g''(1+f'^2),
\end{equation}
for all  $x\in I$, $y\in J$. Here a prime denotes the derivative with respect to each variable. A solution of  Eq. \eqref{ex-1} appears when $f$ and $g$ are linear functions, $f''=g''=0$ identically. Then   $M$ is a plane parallel to the $xy$-plane. We   discard this case by assuming $f''g''\neq 0$ on $I \times J$. Dividing Eq. \eqref{ex-1} by $f''g''$, we obtain
$$2-\frac{1+g'^2}{g''}=\frac{1+f'^2}{f''}.$$
Since the left hand-side depends only on the variable $y$ and the right hand-side on the variable $x$, then we deduce the existence of the nonzero constant $c$ such that 
$$2-\frac{1+g'^2}{g''}=\frac{1}{c}=\frac{1+f'^2}{f''}.$$
Notice that if $c=1/2$, then $1+g'^2=0$, which it is not possible. By solving these equations,  we obtain, up to translations of $x$ and $y$ and suitable constants,  
$$
u(x,y)=-\frac{1}{c}\log \cos (cx)-\frac{2c-1}{c}\log \cos (\frac{cy}{2c-1})) .
$$
 See Fig. \ref{fig1} for the particular case $c=1$. \end{example}

\begin{figure}[hbtp]
\begin{center}
\includegraphics[width=.4\textwidth]{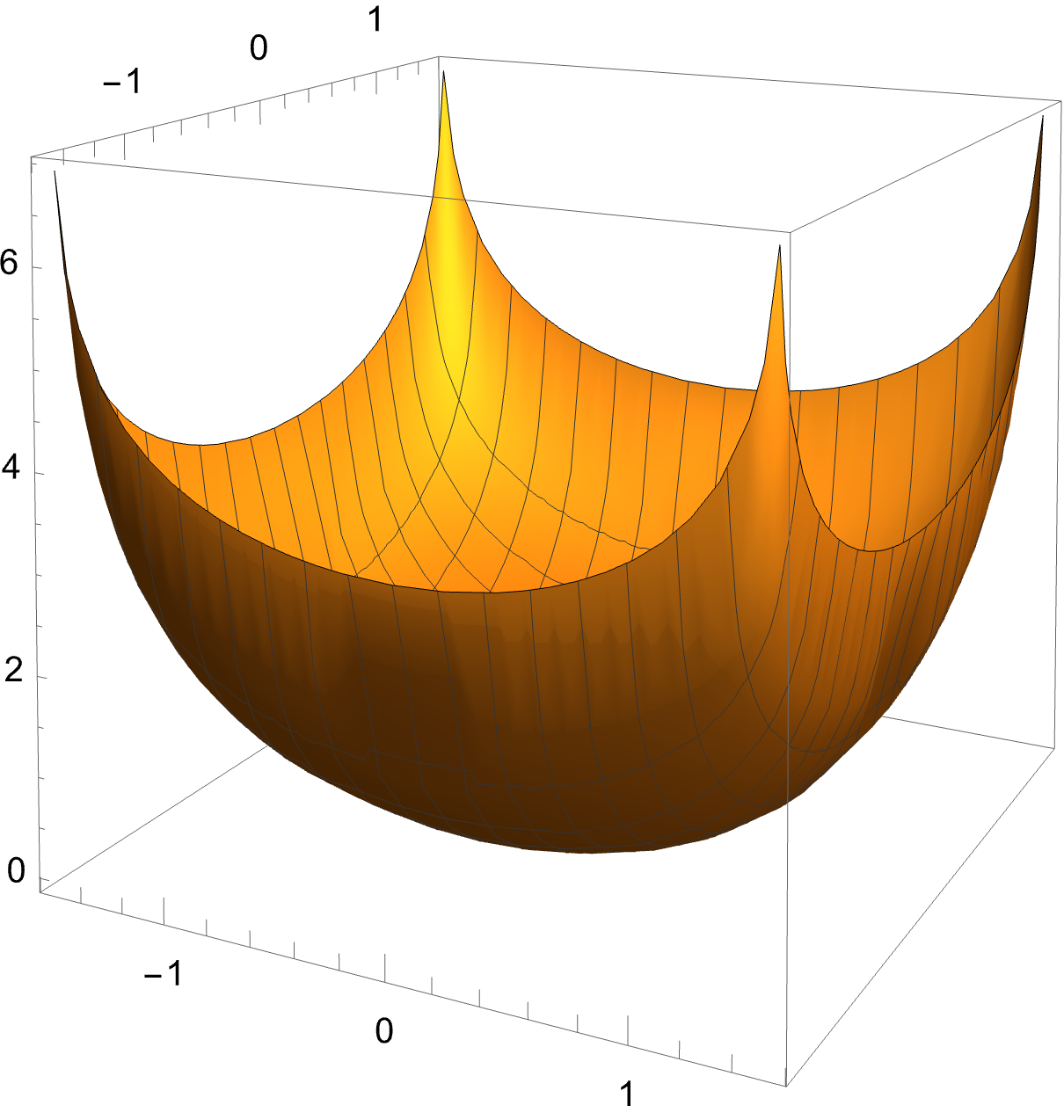}
\end{center}
\caption{Graph of  $z=-\log (\cos (x)\cos (y))$.}\label{fig1}
\end{figure}

\section{Cylindrical surfaces}\label{s3}

 Let $M$ be a cylindrical surface in $\r^3$ whose rulings are parallel to $\vec{w}$, where  $\vec{w}\in \r^3$, $|\vec{w}|=1$. If $\gamma=\gamma(s)$ is the generating curve of $M$ contained in a plane orthogonal to $\vec{w}$, then a parametrization of $M$ is 
\begin{equation} \label{cyl-par}
\psi(s,t)=\gamma(s)+t\vec{w}, \quad s\in I, t\in \r.
\end{equation}
Without loss of generality, we suppose that $\gamma$ is parametrized by arc-length. Let  ${\bf n}$ be the unit normal vector of $\gamma$ and let $\kappa$ be the Frenet curvature of $\gamma$ with $\gamma''=\kappa{\bf n}$.
Since $\gamma$ is contained in a plane orthogonal to $\vec{w}$, consider the orientation on $\gamma$ such that $(\gamma',\vec{w},{\bf n})=1$, where    $(\vec{a},\vec{b},\vec{c})$ stands for the determinant of the matrix formed by three vectors $\vec{a}$, $\vec{b}$, $\vec{c}$ of $\r^3$. 

\begin{theorem}\label{cyl-t2}
Let $\widetilde{\nabla}$ be a canonical snm-connection on $\r^3$.  If  $M$ is a cylindrical surface parametrized by \eqref{cyl-par}, then its sectional curvature $K$ with respect to   $\widetilde{\nabla}$   is 
\begin{equation} \label{cyl-sec2}
 K=\frac{1}{2} \left (\langle \vec{w},\mathsf{C} \rangle^2 + \langle \gamma',\mathsf{C} \rangle^2 -\kappa\langle{\bf n},\mathsf{C}\rangle  \right ).
\end{equation}
\end{theorem}
\begin{proof}
The tangent plane of $M$ is spanned by an orthonormal basis $\{e_1,e_2\}$, where $e_1=\psi_s=\gamma'$ and $e_2=\psi_t=\vec{w}$. We know that the Gaussian curvature is $G=0$. The Gauss map and the mean curvature of $M$ are given by
$$
N= \gamma'\times w , \quad H= \frac{( \gamma',w,\gamma'')}{2 }=\frac{\kappa}{2}.
$$
We compute the covariant derivatives as follows
\begin{equation*}
\begin{array}{lll}
\widetilde{\nabla}_{e_1}e_1=\gamma''+\langle \gamma',\mathsf{C} \rangle \gamma', &\widetilde{\nabla}_{e_1}e_2=\langle \vec{w},\mathsf{C} \rangle\gamma' ,\\
\widetilde{\nabla}_{e_2}e_1=\langle \gamma',\mathsf{C} \rangle\vec{w}, &\widetilde{\nabla}_{e_2}e_2=\langle \vec{w},\mathsf{C} \rangle\vec{w}.
\end{array}
\end{equation*}
Because $[e_1,e_2]=\widetilde{\nabla}^0_{e_1}e_2-\widetilde{\nabla}^0_{e_2}e_1$, we conclude that $[e_1,e_2]=0$. We also compute
\begin{equation*}
\begin{split} 
\widetilde{\nabla}_{e_1}\widetilde{\nabla}_{e_2}e_2&=\langle \vec{w},\mathsf{C} \rangle^2 \gamma',\\
\widetilde{\nabla}_{e_2}\widetilde{\nabla}_{e_1}e_2&=\langle \vec{w},\mathsf{C} \rangle \langle \gamma',\mathsf{C} \rangle \vec{w},\\
\widetilde{\nabla}_{e_2}\widetilde{\nabla}_{e_1}e_1&=( \langle \gamma'',\mathsf{C} \rangle+\langle \gamma',\mathsf{C} \rangle^2) \vec{w},\\\widetilde{\nabla}_{e_1}\widetilde{\nabla}_{e_2}e_1&=\langle \gamma'',\mathsf{C} \rangle\vec{w}+\langle \gamma',\mathsf{C} \rangle\langle \vec{w},\mathsf{C} \rangle \gamma',
\end{split}
\end{equation*}
and thus
$$
\widetilde{K}=\frac{1}{2} \left (\langle \vec{w},\mathsf{C} \rangle^2 + \langle \gamma',\mathsf{C} \rangle^2  \right ).
$$
By \eqref{help0} we find
$$K=\frac{1}{2} \left (\langle \vec{w},\mathsf{C} \rangle^2 + \langle \gamma',\mathsf{C} \rangle^2 -(\gamma',\vec{w},\mathsf{C})(\gamma'',\gamma',\vec{w})  \right ).
$$
The result follows because $\gamma'\times\vec{w}={\bf n}$ and $\gamma''=\kappa{\bf n}$.
\end{proof}

We distinguish two particular cases, when the rulings are parallel or orthogonal to the constant vector field $\mathsf{C}$.

\begin{corollary} \label{cor32}
Any cylindrical surface whose rulings are parallel to $\mathsf{C}$ has constant sectional curvature  $K=1/2$ with respect to a canonical snm-connection determined by $\mathsf{C}$.
\end{corollary}

Suppose that the rulings are orthogonal to $\mathsf{C}$. In the next result we are going to obtain explicit parametrizations of cylindrical surfaces with constant sectional curvature. Without loss of generality, we suppose that $\mathsf{C}=\partial_z$ and $\vec{w}=(0,1,0)$.  Then $\gamma$ is contained in the $xz$-plane, say $\gamma(s)=(x(s),0,z(s))$, for smooth functions $x,z\colon I\to\r$. The case that $M$ is a plane is particular. Any plane of $\r^3$ perpendicular to $\partial_z$ can be viewed as a cylindrical surface with rulings orthogonal to $\partial_z$. By Prop. \ref{pr-21}, we know that its curvature $K$ is constant with $0\leq K\leq \frac12$.  We discard this case. 

\begin{corollary} \label{c-cyl}
Let $\widetilde{\nabla}$ be the canonical snm-connection determined by $\partial_z$ and $M$ be a non-planar cylindrical surface whose rulings are orthogonal to $\partial_z$. If  the sectional curvature $K$ with respect to $\widetilde{\nabla}$ is constant, then the parametrization of the generating curve $\gamma$ is
\begin{enumerate}
\item Case $K>0$, then $\gamma(s)=(\int^s \sqrt{1-2 K \tanh ^2\left(\sqrt{2 K} t\right)}\, dt ,-\log(\cosh(\sqrt{2K}s)))$.
\item Case $K=0$, then $\gamma(s)=(\pm \tan ^{-1}\left(\sqrt{s^2-1}\right)-\sqrt{s^2-1},-\log(s))$.
\item Case $K<0$, then  $\gamma(s)=(\int^s \sqrt{1-2 K \tan ^2\left(\sqrt{-2 K} t\right)}\, dt,-\log(\cos(\sqrt{-2K}s))$.
\end{enumerate}
\end{corollary}

\begin{proof}
Since $\gamma$ is parametrized by arc-length, we know $x'^2+z'^2=1$ and $\gamma'=(x',0,z')$. By the choice of orientation on $\gamma$ given in Thm. \ref{cyl-t2}, the normal vector is ${\bf n}=(-z',0,x')$. Identity \eqref{cyl-sec2} is 
$$z''=z'^2-2K.$$
The solution of this equation depends on the sign of $K$. Up to an additive constant on the functions $x$ and $z$ as well as in the parameter $s$, which it is only a translation of the surface (Rem. \ref{re-23}), we have
\begin{enumerate}
\item $K>0$; then $z(s)=-\log(\cosh(\sqrt{2K}s))$.
\item $K=0$; then $z(s)=-\log(s)$.
\item $K<0$; then  $z(s)=-\log(\cos(\sqrt{-2K}s))$.
\end{enumerate}
The result follows from the identity $x'^2+z'^2=1$.
\end{proof}

In Fig. \ref{fig2} we depict some graphics of the generating curves for different values of $K$. Notice that the domain of $\gamma$ is not $\r$ in general because the root that appears in the integrals that define  the $x$-coordinate of $\gamma$. For example, if $K=0$, then $s\in [1,\infty)$. 

\begin{figure}[hbtp]
\begin{center}
\includegraphics[width=.3\textwidth]{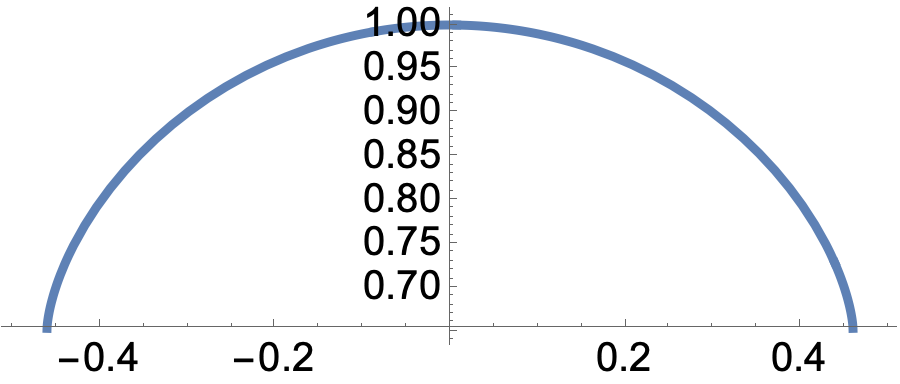}\quad \includegraphics[width=.3\textwidth]{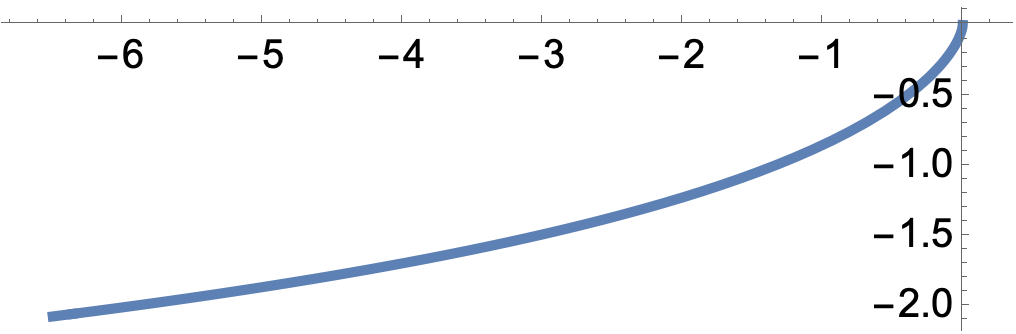}\quad\includegraphics[width=.3\textwidth]{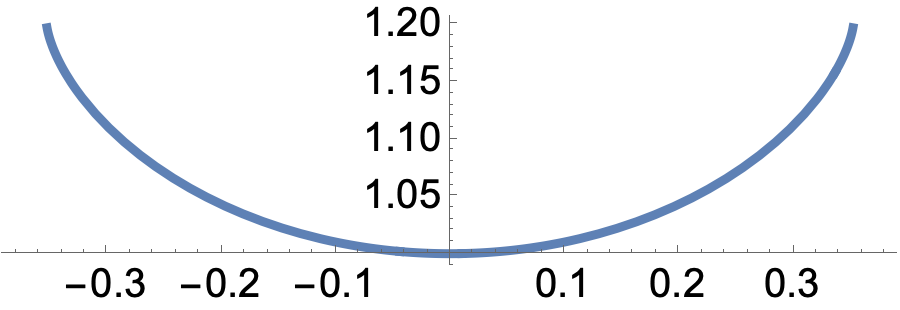}
\end{center}
\caption{Graphics of generating curves of Cor. \ref{c-cyl}: $K=1$ (left), $K=0$ (middle) and $K=-1$ (right). }\label{fig2}
\end{figure}

It is worth to consider the cases $K=  1/2$ and $K=-1/2$. In such a case, the integrals of Cor. \ref{c-cyl} can be explicitly solved.

\begin{corollary}\label{cor-34}
Let $\widetilde{\nabla}$ be the canonical snm-connection determined by $\partial_z$, $M$ a non-planar cylindrical surface whose rulings are orthogonal to $\partial_z$, and $K$ the sectional curvature of $M$ with respect to $\widetilde{\nabla}$.
\begin{enumerate}
\item If $K=\frac12$, then $\gamma(s)=(\tan ^{-1}(\sinh (s)),-\log  \cosh (s) )$.
\item If $K=-\frac12$, then $\gamma(s)=(\sqrt{2} \sin ^{-1}\left(\sqrt{2} \sin (s)\right)-\cot ^{-1}\left(\sqrt{\cot(s)^2-1}\right),-\log  \cos (s) )$.
\end{enumerate} 
\end{corollary}

For $K=1/2$, the curve $\gamma$ in (1) is called grim reaper. The usual parametrization of the grim reaper is $y(x)=-\log(\cos(x))$ in the $(x,y)$-plane $\r^2$. This  is deduced immediately by letting $x=\tan ^{-1}(\sinh (s))$. The grim reaper is a remarkable curve in the theory of curve-shortening flow  \cite{hal}.

\section{Rotational surfaces }\label{s4}

In this section we study rotational surfaces with constant sectional curvature. A first problem is the relation between the axis $L$ of the surface and the vector field $\mathsf{C}$ that defines the canonical snm-connection. As we said in the Introduction, there is no {\it a priori} a relation between both. However, we prove that they must be parallel.

\begin{theorem}\label{t1}
Let $\widetilde{\nabla}$ be a canonical snm-connection on $\r^3$ determined by the vector field $\mathsf{C}$ and $M$ be  a  rotational  surface in $\r^3$ about an axis $L$. If $M$ has constant sectional curvature $K$, then either  $M$ is any plane   and $K\geq 0$     or $L$ is parallel to $\mathsf{C}$.  
\end{theorem}

\begin{proof}
After a change of coordinates in $\r^3$, we can suppose that  the axis $L$ of $M$ is the $z$-axis. Let $\mathsf{C}=a\partial_x+b\partial_y+c\partial_z$, for $a,b,c \in \r$. Let also $\gamma\colon I\subset\r\to\r^3$ be the generating curve of $M$ which we can assume that it is contained in the $xz$-plane, namely,  
$$
\gamma(s)= (x(s),0,z(s)), \quad s\in I\subset \r.
$$
We also assume that $\gamma$ is   parametrized by arc-length, that is, $x'^2+z'^2=1$. Let $\kappa=x'z''-z'x''$ be its Frenet  curvature with respect to the induced Levi-Civita connection $\widetilde{\nabla}^0$.  A parametrization of $M$ is 
$$
\psi(s,t)=( x(s)\cos t,x(s)\sin t,z(s)),\quad s\in I, t\in\r.
$$
For the computation of $K$, we calculate all terms of \eqref{help0}. 
The tangent plane of $M$ is spanned by $\{e_1,e_2\}=\{\psi_s,\psi_t\}$, where
\begin{equation*}
\begin{split}
e_1&=( x'\cos t,x'\sin t,z'),\\
e_2&=( -x\sin t,x\cos t,0).
\end{split}
\end{equation*}
 The coefficients of the first fundamental form are  $g_{11}=1$, $g_{12}=0$ and $g_{22}=x^2$.   The unit normal vector of $M$ is 
$$N=(-z'\cos t,-z'\sin t,x').$$
Then it is immediate
\begin{equation}\label{hk}
\begin{split}
\langle \mathsf{C},N\rangle&=-az'\cos t-b z'\sin t+c x',\\
 H&=\frac{1}{2x}(z'+x\kappa),\\
  G&=\frac{z'\kappa}{x}.
 \end{split}
 \end{equation}
 We now calculate $\widetilde{K}$. For this we employ the definition \eqref{sc} taking into account that now the denominator   is $2(g_{11}g_{22}-g_{12}^2)=2x^2$. We begin computing   the covariant derivatives $\widetilde{\nabla}_{e_i}e_j$, $1\leq i,j\leq 2$. From \eqref{sc}, we have 
\begin{eqnarray*}
\widetilde{\nabla}_{e_1}e_1&=&\psi_{ss}+\langle \mathsf{C},e_1\rangle e_1,\\
\widetilde{\nabla}_{e_1}e_2&=&\psi_{st}+\langle \mathsf{C},e_2\rangle e_1 ,\\
\widetilde{\nabla}_{e_2}e_1&=&\psi_{st}+\langle \mathsf{C},e_1\rangle e_2 ,\\
\widetilde{\nabla}_{e_2}e_2&=&\psi_{tt}+\langle \mathsf{C},e_2\rangle e_2.
\end{eqnarray*}

Similarly, the covariant derivatives of second order are calculated. We obtain
\begin{eqnarray*}
\widetilde{\nabla}_{e_1}\widetilde{\nabla}_{e_2}e_2&=&(\widetilde{\nabla}_{e_2}e_2)_s+\langle \mathsf{C},\widetilde{\nabla}_{e_2}e_2\rangle e_1,\\
\widetilde{\nabla}_{e_2}\widetilde{\nabla}_{e_1}e_2&=&(\widetilde{\nabla}_{e_1}e_2)_t+\langle \mathsf{C},\widetilde{\nabla}_{e_1}e_2\rangle e_2 ,\\
\widetilde{\nabla}_{e_2}\widetilde{\nabla}_{e_1}e_1&=&(\widetilde{\nabla}_{e_1}e_1)_t+\langle \mathsf{C},\widetilde{\nabla}_{e_1}e_1\rangle e_2,\\
\widetilde{\nabla}_{e_1}\widetilde{\nabla}_{e_2}e_1&=&(\widetilde{\nabla}_{e_2}e_1)_s+\langle \mathsf{C},\widetilde{\nabla}_{e_2}e_1\rangle e_1.
\end{eqnarray*}
Obviously, $[e_1,e_2]=0$. The curvature $\widetilde{R}$ is
\begin{equation*}
\begin{split}
\widetilde{R}(e_1,e_2,e_2,e_1)&=\langle(\widetilde{\nabla}_{e_2}e_2)_s-(\widetilde{\nabla}_{e_1}e_2)_t,e_1\rangle+\langle \mathsf{C}, \widetilde{\nabla}_{e_2}e_2\rangle\\
&= x^2 (b \cos t-a \sin t)^2,\\
 \widetilde{R}(e_2,e_1,e_1,e_2)&=\langle(\widetilde{\nabla}_{e_1}e_1)_t-(\widetilde{\nabla}_{e_2}e_1)_s,e_2\rangle+x^2\langle \mathsf{C}, \widetilde{\nabla}_{e_1}e_1\rangle\\
&=x^2 \left(x' (a \cos t+b \sin t)+c z'\right)^2.
\end{split}
\end{equation*}
 This gives
\begin{equation*}
\begin{split}
\widetilde{K}&=\frac{ \widetilde{R}(e_1,e_2,e_2, e_1)+\widetilde{R}(e_2,e_1,e_1, e_2) }{2x^2}\\
=&\frac{1}{2}\left((b \cos t-a \sin t)^2+(x' (a \cos t+b \sin t)+c z')^2\right).
\end{split}
\end{equation*}
Finally, using  \eqref{help0}, we obtain
 
\begin{equation*}
\begin{split}
K&= \frac{1}{2}\left((b \cos t-a \sin t)^2+(x' (a \cos t+b \sin t)+c z')^2\right)\\
&+G-(c x'-z' (a \cos t+b \sin t))H.
\end{split}
\end{equation*}

The above expression can be written as a polynomial equation of type 
$$\sum_{n=0}^2(A_n(s)\cos (nt)+B_n(s)\sin (nt))=0.$$ 
Since the functions $\{ \cos(nt),\sin(nt)\}$, $0\leq n\leq 2$, are linearly independent, then all coefficients $A_n$ must vanish identically. The computation of these coefficients yields
\begin{equation*}
\begin{split}
A_2&=\frac{(b^2-a^2)z'^2}{4},\\
B_2&=-\frac12 ab z'^2,\\
A_1&=az'(H+cx'),\\
B_1&=bz'(H+cx'),\\
A_0&=\frac{1}{4} \left(a^2+b^2+2 c^2\right) z'^2+\frac{1}{2} \left(a^2+b^2\right) x'^2-c H x'-K+G.
\end{split}
\end{equation*}
From $A_2=0$ and $B_2=0$, we have the following discussion of cases.
\begin{enumerate}
\item Case $z'=0$ identically. Then $z$ is a constant function and this implies that $M$ is a horizontal plane. In particular, $x'^2=1$. Without loss of generality, we suppose   $x(s)=s$. Since $G=H=0$, equation $A_0=0$ is simply 
$$K=\frac12(a^2+b^2).$$
This proves the result in this case.  
\item Case that $z'(s)\not=0$ at some value $s$. Then $z'\not=0$ around $s$ and $A_2=B_2=0$ implies $a=b=0$. Thus $\mathsf{C}=\pm  \partial_z$ and this proves that $\mathsf{C}$ is parallel to the $z$-axis, which it is the rotation axis of $M$.
\end{enumerate}
\end{proof}

 Once proved Thm. \ref{t1}, we can suppose that the vector field $\mathsf{C}$ is $\partial_z$ and $M$ is a rotational surface about the $z$-axis.  Following the proof of that theorem,  all coefficients $A_n$ and $B_n$, $1\leq n\leq 2$ are trivially $0$ except $A_0$ which it is 
$$K=\frac{z'^2}{2}+G-x'H.$$
Using the value of $G$ and $H$ given in \eqref{hk}, the above equation gives us the expression of $K$ of a rotational surface in terms of its generating curve, namely,     
\begin{equation} \label{Kgf}
K=\frac{1}{2x}((2z'-xx')\kappa+z'(xz'-x')).
\end{equation}

We study when the parenthesis of the right hand-side of \eqref{Kgf} are $0$ identically. 

\begin{proposition} \label{neq}
If $K$ is constant in \eqref{Kgf}, then the functions $2z'-xx'$ and $xz'-x'$ cannot vanish identically in $I$.  
\end{proposition}
\begin{proof}
\begin{enumerate}
\item Case $2z'-xx'=0$. Then neither $x'$ nor $z'$ can vanish identically. From \eqref{Kgf}, we have
$$
x'^2=\frac{8K}{x^2-2}.
$$
Since  $x'^2+z'^2=1$, then 
$$
x'^2=\frac{4}{4+x^2}.
$$
Combining both equations, we get $8K+(2K-1)x^2+2=0$, then $x$ is a constant function, which it is a contradiction.

\item Case $ xz'-x'=0$.   Since $z'=x'/x$, then $\kappa=   -\frac{x'^3}{x^2}$. Thus  \eqref{Kgf} is 
$$2Kx^2=x'^4(x-2). $$
 On the other hand, it follows $x'^2+z'^2=1$ that
$$
x'^2=\frac{x^2}{1+x^2}.
$$
Combining both equations we obtain that $x=x(s)$ is a constant function. From $ xz'-x'=0$, we have $z$ constant too, which it is a   contradiction by regularity of $\gamma$.
\end{enumerate}
\end{proof}

In the following two results we study the case when the generating curve $\gamma$ of $M$ has constant curvature $\kappa$, that is, $\gamma$ is a straight-line and a circle. First, suppose that $\gamma$ is a straight-line. This implies that $M$ is a conical rotational surface. 

\begin{theorem}\label{t-s} Let $\widetilde{\nabla}$ be a canonical snm-connection and $M$ be a rotational surface about the $z$-axis. Assume that the sectional curvature $K$ of $M$ with respect to $\widetilde{\nabla}$ is constant. If the generating curve of $M$ is a straight-line, then either $M$ is a circular cylinder and $K=1/2$, or $M$ is a horizontal plane and $K=0$.
\end{theorem}

\begin{proof} We follow the notation of Thm. \ref{t1}. Since $\gamma$ is parametrized by arc-length, then   there is a real number $\theta\in\r$ such that $\gamma$ can be written as 
$$\gamma(s)=(c_1,c_2)+(\cos\theta,\sin\theta)s,\quad c_1,c_2\in\r.$$
Equation \eqref{Kgf} is now
$$2K(s\cos\theta+c_1)-\sin^2\theta(s\cos\theta+c_1)+\sin\theta\cos\theta=0.$$
This  is a polynomial equation on $s$, so all coefficients must vanish. Therefore
\begin{eqnarray*}
(2K-\sin^2\theta)\cos\theta&=&0, \\
(2K-\sin^2\theta)c_1+\sin\theta\cos\theta&=&0.
\end{eqnarray*}
\begin{enumerate}
\item Case $\cos\theta=0$. Then $\gamma(s)=(c_1,\pm s+c_2)$. In particular $c_1>0$. This implies that $M$ is  a circular cylinder of radius $\sqrt{c_1}$. The second equation gives $K=1/2$.

\item Case $\cos\theta\neq 0$. Then  $2K-\sin^2\theta=0$ and the second equation gives $\sin\theta=0$. Thus $\gamma(s)=(\pm s+c_1,c_2)$ and $M$ is a horizontal plane of equation $z=c_2$. Here $K=0$.  
\end{enumerate}
\end{proof}

Finally, we suppose that $\gamma$ is a circle. This implies that $M$ is torus of revolution or a rotational ovaloid. 
\begin{theorem}\label{pr-c}
Let $\widetilde{\nabla}$ be a canonical snm-connection and $M$ be a rotational surface about the $z$-axis. Assume that the sectional curvature $K$ of $M$ with respect to $\widetilde{\nabla}$ is constant. Then the generating curve of $M$ cannot be a circle.  
\end{theorem}

\begin{proof}
By contradiction, suppose that $\gamma$ is a circle of radius $r>0$. A parametrization of $\gamma$ is 
$$\gamma(s)=(c_1,c_2)+r\left(\cos (s/r),\sin (s/r)\right).$$
Substituting into \eqref{Kgf}, we obtain
\begin{equation*}
\begin{split}
2K(c_1+r\cos(s/r))&-\frac{1}{r}\left(2\cos (s/r)+(c_1+r\cos (s/r))\sin(s/r)\right)\\
&-\cos(s/r)\left((c_1+r\cos (s/r))\cos(s/r)+\sin(s/r)\right)=0.
\end{split}
\end{equation*}
This equation writes as 
$$\sum_{n=0}^3 (A_n\cos(s/r)+B_n\sin(s/r))=0,$$
where $A_n$ and $B_n$ are real constants. Since all $A_n$ and $B_n$ must $0$,  a computation gives $A_3=-\frac{r}{3}$, obtaining a contradiction. 
\end{proof}

The study of solutions of  \eqref{Kgf} is difficult to do in all its generality and Thms. \ref{t-s} and \ref{pr-c} are the first results. An interesting value for $K$ is $1/2$ because this is the curvature of a circular cylinder (for any radius) and that of a plane parallel to  $\partial_z$. If $K=1/2$, then Eq. \eqref{Kgf} is 
\begin{equation}\label{kk}
\kappa=\frac{x'(x+z')}{2z'-xx'}.
\end{equation}
An interesting question is if this equation has a solution for curves starting orthogonally from the rotation axis. If $s=0$ is the time where $\gamma$ intersects the $z$-axis, then we need $x(0)=0$ and $z'(0)=0$.  However, the left hand-side of \eqref{kk} is not defined at $s=0$. This implies that existence of such solutions is not assured. We prove that these solutions, indeed, exist.

\begin{theorem} \label{ort} There exist rotational surfaces with constant sectional curvature $K=1/2$ intersecting orthogonally the rotation axis.
\end{theorem}

\begin{proof}
For our convenience, we work assuming that $\gamma$ is locally a graph $z=z(x)$. Then 
$$x'(s)=\frac{1}{\sqrt{1+z'(x)^2}},\quad z'(s)=\frac{z'(x)}{\sqrt{1+z'(x)^2}},\quad \kappa=\frac{z''(x)}{(1+z'(x)^2)^{3/2}}.$$
Then \eqref{kk} becomes
$$\frac{z''}{(1+z'^2)^{3/2}}=\frac{x\sqrt{1+z'^2}+z'}{(2z'-x)\sqrt{1+z'^2}},$$
or equivalently, 
$$(2z'-x)\frac{z''}{(1+z'^2)^{3/2}}=\frac{x\sqrt{1+z'^2}+z'}{ \sqrt{1+z'^2}}.$$
This equation also writes as 
$$\frac{d}{dx}\left((2z'-x)\frac{z'}{\sqrt{1+z'^2}}\right)=\frac{d}{dx}(2\sqrt{1+z'^2}+\frac{x^2}{2}).$$
Thus there is an integration constant $c\in\r$ such that
$$ (2z'-x)\frac{z'}{\sqrt{1+z'^2}} = 2\sqrt{1+z'^2}+\frac{x^2}{2}+c.$$
If $\gamma$ intersects orthogonally the $z$-axis, then we have $z'(0)=0$. This gives $c=-2$, obtaining a first integration of \eqref{Kgf}, namely, 
$$ (2z'-x)\frac{z'}{\sqrt{1+z'^2}} = 2(\sqrt{1+z'^2}-1)+\frac{x^2}{2}.$$
Squaring both sides of this equation, we get
$$
(4x^2-(x^2-4)^2)z'^2+16xz'+16-(x^2-4)^2=0.
$$
By standard theory of existence of ODE, this equation has a solution with initial value $z'(0)=0$, proving the result. 
\end{proof}

\section*{Acknowledgements}
Rafael L\'opez  is a member of the IMAG and of the Research Group ``Problemas variacionales en geometr\'{\i}a'',  Junta de Andaluc\'{\i}a (FQM 325). This research has been partially supported by MINECO/MICINN/FEDER grant no. PID2020-117868GB-I00,  and by the ``Mar\'{\i}a de Maeztu'' Excellence Unit IMAG, reference CEX2020-001105-M, funded by MCINN/AEI/10.13039/501100011033/ CEX2020-001105-M.



\begin{thebibliography}{99}

\bibitem{ag} N. S. Agashe, A semi-symmetric non-metric connection on a Riemannian manifold. Indian J. Pure Appl. Math. 23 (1992), 399--409.

\bibitem{ac} N. S. Agashe, M. R. Chafle, On submanifolds of a Riemannian manifold with a semi-symmetric non-metric connection. Tensor 55 (1994), 120--130.

\bibitem{che} B.-Y. Chen, Total mean curvature and submanifolds of finite type. Second edition. With a foreword by Leopold Verstraelen. Series in Pure Mathematics, 27. World Scientific Publishing Co. Pte. Ltd., Hackensack (2015).

\bibitem{fs} A. Friedmann, J. A. Schouten, \"{U}ber die Geometrie der halbsymmetrischen \"{U}bertragungen. Math. Z. 21 (1924), 211--223.

 
 

\bibitem{hal} H. P. Halldorsson, Self-similar solutions to the curve shortening flow. Trans. Amer. Math. Soc. 364 (2012),  5285--5309.
\bibitem{hay} H. Hayden, Subspaces of a space with torsion. Proc. London Math. Soc. 34 (1932), 27--50.

\bibitem{im} T. Imai, Hypersurfaces of a Riemannian manifold with semi-symmetric metric connection. Tensor (N. S.) 23 (1972), 300--306. 

\bibitem{lyl} C. W. Lee, D. W. Yoon, J. W. Lee, Optimal inequalities for the Casorati curvatures of submanifolds of real space forms endowed with semi-symmetric metric connections. J. Inequal. Appl. 2014, 327 (2014).



 
 \bibitem{mc1} A. Mihai, C. \"{O}zg\"{u}r, Chen inequalities for submanifolds of real space forms with a semi-symmetric metric connection. Taiwan. J. Math. 4 (2010), 1465--1477.

 \bibitem{mc2} A. Mihai and C. \"{O}zg\"{u}r, Chen inequalities for submanifolds of complex space forms and Sasakian space forms endowed with semi-symmetric metric connections. Rocky Mountain J. Math. 41 (2011), 1653--1673.

\bibitem{am0} A. Mihai, I. Mihai, A note on a well-defined sectional curvature of a semi-symmetric non-metric connection. Int. Electron. J. Geom. 17(1) (2024), 15-23.

\bibitem{na} Z. Nakao, Submanifolds of a Riemannian manifold with semisymmetric metric connections. Proc. Amer. Math. Soc. 54 (1976), 261--266.

 
\bibitem{wa} Y. Wang,  
Minimal translation surfaces with respect to semi-symmetric connections in $R^3$ and $R^3_1$. 
Bull. Korean Math. Soc. 58 (2021),  959--972.

\bibitem{ya} K. Yano, On semi symmetric metric connection. Rev. Roum. Math. Pures Appl. 15 (1970), 1579--1591.

\end{thebibliography}
 \end{document}